\documentclass[11pt]{amsart}
\usepackage{amsmath}
\usepackage{dsfont}

\theoremstyle{plain}
\newtheorem{theorem}{Theorem}[section]
\newtheorem{lemma}[theorem]{Lemma}
\newtheorem{prop}[theorem]{Proposition}

\theoremstyle{definition}
\newtheorem{remark}[theorem]{Remark}

\numberwithin{equation}{section}
\def\be{\begin{equation}}
\def\ee{\end{equation}}

\begin{document}

\title [Boundary Expansions of Complete Conformal Metrics]
{Boundary Expansions of Complete Conformal Metrics with Negative Ricci Curvatures}

\author{Yue Wang}
\address{School of Mathematical Sciences\\
Peking University\\
Beijing, 100871, China}  \email{yuewang37@pku.edu.cn}
\begin{abstract}
 We study the boundary behaviors of a complete conformal metric which solves the $\sigma_k$-Ricci problem on the interior of a manifold with boundary. We establish asymptotic expansions and also $C^1$ and $C^2$ estimates for this metric multiplied by the square of the distance in a small neighborhood of the boundary.
\end{abstract}
\thanks{The author acknowledges the support of NSFC Grant NSFC11571019.}
%\date{\today}
\maketitle

\section{Introduction}

Let $(M, \partial M, g)$ be a smooth Riemannian manifold with boundary and $1\le k\le n$.
We consider the following problem:
\begin{align}\label{M.G}
\sigma_k[-g^{-1}Ric(e^{2u}g)]&=(n-1)^kC_n^ke^{2ku}\quad\text{in } M\setminus\partial M,\\
\label{M.G.bdry}
u&=\infty \quad\text{on } \partial M,
\end{align}
where $C_n^k=\binom{n}{k}$, $Ric(e^{2u}g)$ is the Ricci curvature of the conformal metric $e^{2u}g$,
and $\sigma_k(A)$ is the $k$-th elementary symmetric polynomial
in the eigenvalues of the symmetric matrix $A$. Let $\Gamma_k^+$
be the connected component of the set $\{\sigma_k > 0\}$ which contains the positive definite cone.

Gursky, Streets and Warren \cite{M.Gursky1} proved that \eqref{M.G} and \eqref{M.G.bdry} admit
a unique solution $u\in C^\infty(M\setminus \partial M)$ with an additional requirement that
$-Ric(e^{2u}g)\in \Gamma_k^+$. Moreover, $e^{2u}g$ is a complete metric and
\begin{align}\label{G-expansion}
    \lim_{x\rightarrow \partial M}\big[u+\log d\big]=0,
\end{align}
where $d$ is the distance to $\partial M$. Refer to Theorem 1.4 in \cite{M.Gursky1}.
By comparing \eqref{M.G} with the equation in Theorem 1.4 \cite{M.Gursky1}, we note that
a constant $(n-1)^kC_n^k$ is inserted in the right-hand side of \eqref{M.G}. With the newly inserted
constant factor,  the constant term in the expansion \eqref{G-expansion} is zero.

In this paper, we study further expansions of $u$ near the boundary. For brevity, we consider the case
that $g$ is the standard Euclidean metric.
Assume $\Omega\subseteq \mathbb R^n$ is a bounded smooth domain, for $n\ge 3$. For $u\in C^2(\Omega)$, define
a symmetric matrix $A(u)$ by
\begin{align}\label{Aij}
A(u)=(n-2)\nabla^2u+\Delta u I_{n\times n}+(n-2)[|\nabla u|^2I_{n\times n}-\nabla u\otimes\nabla u],
\end{align}
where $I_{n\times n}$ is the identity $n\times n$ matrix. We are led to the following problem:
\begin{align}\label{main}
\sigma_k(A(u))&=(n-1)^kC_n^ke^{2ku}\quad \text{in } \Omega,\\
\label{mainbdry}
u&=\infty\quad \text{on } \partial\Omega,
 \end{align}
with the additional requirement that $A(u)\in \Gamma_k^+$.

Set $$e^{2u}=w^{\frac{4}{n-2}}.$$
For $k=1,$
\eqref{main} and \eqref{mainbdry} are reduced to the following more familiar form:
\begin{align}\label{ln}
\begin{split}
 \Delta w&=\frac{1}{4}n(n-2)w^{\frac{n+2}{n-2}}
 \quad \text{in } \Omega,\\
   w&=\infty\quad \text{on }  \partial\Omega.
 \end{split}
\end{align}
Loewner and Nirenberg \cite{LN} proved the existence of the unique positive solution of  \eqref{ln}
and Aviles and McOwen \cite{AM} proved the same result for the corresponding equation in general manifolds.
Andersson, Chru\'sciel and Friedrich \cite{ACF} and Mazzeo \cite{M} established
the polyhomogeneous expansions for the solutions.
Graham \cite{RG} studied the renormalized volume expansion.
He identified the first two renormalized volume coefficients and the information contained in the anomaly,
namely, the difference of the renormalized volumes corresponding to different choices of conformal representatives,
and proved the conformal invariance of the energy, the coefficient of the $\log$-term in the volume expansion.

We now present our main results for \eqref{main} and \eqref{mainbdry}. As in \eqref{G-expansion},
we denote by $d$ the distance function in $\Omega$ to $\partial \Omega$ and set
\begin{align}\label{Ddelta}
 D_\delta=\{x\in\Omega|d(x)\leq \delta\}\cap\Omega.
\end{align}
If $\partial \Omega$ is $C^\infty,$ then $d$ is $C^\infty$ in a sufficiently small neighborhood of $\partial \Omega.$
In this paper, we use the principal coordinates in $D_\delta$ and denote by $(x',d)$ the points in $D_\delta$,
for $\delta>0$ sufficiently  small.

We have the following result for the expansions of $u+\log d$ up to the first $\log$-term with the coefficients
in terms of principal curvatures and their derivatives.

\begin{theorem}\label{expansion}
Assume that $\Omega$ is a bounded smooth domain in $\mathbb{R}^{n},$ for $n\ge 3$,
and that $u$ is the solution of  \eqref{main}-\eqref{mainbdry}. Then,
 \begin{align}\label{e-d2}
   |u+\log d-c_1d-\cdots-c_{n-1}d^{n-1}-c_{n,1}d^n\log d|\leq C d^n\quad\text{in }  D_{\delta_2},
 \end{align}
where $C$ and $\delta_2$ are positive constants depending only on $\Omega, \,n$ and $k$,
 and  $c_1,\cdots, c_{n-1}$ and $c_{n,1}$ are smooth functions on $\partial\Omega$.
 \end{theorem}

We note that $c_1,\cdots, c_{n-1}$ and $c_{n,1}$ will be given by (\ref{c1}), (\ref{ci}), and (\ref{cn1}).

We point out that Theorem \ref{expansion}  holds for solutions of \eqref{M.G} and \eqref{M.G.bdry}, not just
for those of \eqref{main} and \eqref{mainbdry}. To emphasize the dependence of solutions on $k$,
we denote by $u_k$  the solution of \eqref{main} and \eqref{mainbdry}  and write its expansion as
 \begin{align}\label{cjk}
   w_k=-\log d+c_1^kd+c_2^kd^2+\cdot\cdot\cdot+c_{n-1}^kd^{n-1}+c_{n,1}^kd^n\log d+O(d^n).
\end{align}
Denote by $g_0$ the Euclidean metric.
As mentioned earlier, $u_k$ is the solution of \eqref{M.G} and \eqref{M.G.bdry} when $g=g_0.$
Consider a conformal metric $g=e^{2\rho}g_0$ and the corresponding solution $\widetilde{w}_k$
of \eqref{M.G} and \eqref{M.G.bdry}. Then, $\widetilde{w}_k$
also has an expansion in the form
\begin{align}\label{cjktilda}
    \widetilde{w}_k=-\log d+\widetilde{c}_1^kd+\widetilde{c}_2^kd^2
    +\cdot\cdot\cdot+\widetilde{c}_{n-1}^kd^{n-1}+\widetilde{c}_{n,1}^kd^n\log d+O(d^n),
\end{align}
In fact, we can take $\widetilde{w}_k+\rho$ as a function and then apply the uniqueness result for the Euclidean metric $g_0.$
A direct consequence of Theorem \ref{expansion} is the following result, which we can compare
with results in \cite{RG}.

\begin{prop}\label{conformal-diff}
Let
$c_j^k$, $j=1,\cdots,n-1$, be the terms as in \eqref{cjk}. Then,
\begin{align}\label{dif-invar}
 c_j^k-c_j^1 \text { is conformally invariant;}
\end{align}
namely,
\begin{align*}
 c_j^k-c_j^1 =\widetilde{c}_j^k-\widetilde{c}_j^1,
\end{align*}
where $\widetilde{c}_j^k$ is given in \eqref{cjktilda}, $j=1,\cdots,n-1$.
\end{prop}

Next, we derive the $C^1$ and
$C^2$ estimates for $u+\log d$.

\begin{theorem}\label{nabla}
Assume that $\Omega$ is a bounded smooth domain in $\mathbb{R}^{n},$ for $n\ge 3$,
     and that $u$ is the solution of  \eqref{main}-\eqref{mainbdry}. Then,
\begin{align*}
    |\nabla(u+\log d -c_1d)|\leq Cd^\alpha\quad\text{in }  D_{\delta_3},
\end{align*}
where $C$ and $\delta_3$ are positive constants depending only on $ \Omega,$ $n$ and $k$,
$c_1$ is the function in \eqref{c1}, $\alpha={1}/{2}$ when $n=3$ and  $\alpha=1$  when $n\ge4$.
\end{theorem}

\begin{theorem}\label{C2}
Assume that $\Omega$ is a bounded smooth domain in $\mathbb{R}^{n},$ for $n\ge 8,$
     and that $u$ is the solution of  \eqref{main}-\eqref{mainbdry}. Then,
\begin{align*}
    |\nabla^2(u+\log d -c_1d)|\leq C\quad\text{in }  D_{\delta_4},
\end{align*}
 where $C$ and $\delta_4$ are positive constants depending only on $ \Omega,$ $n$ and $k$,
 and $c_1$ is the function in \eqref{c1}.
\end{theorem}

The paper is organized as follows.
In Section 2, we prove the boundary expansion of $u+\log d$ and
the conformal invariance of the difference of corresponding coefficients in the expansions.
In Section 3 and Section 4, we derive the $C^1$ and
$C^2$ estimates for $u+\log d$, respectively.

The author would like to thank Professor Matthew Gursky and Professor Qing Han for suggesting the problem studied in this
paper and many helpful discussions. The author is very grateful to
them for their stimulating ideas and great encouragement.

\section{Boundary Expansions}

Consider the operators
 \begin{align}\label{main1}
F(u)=\sigma_k(\lambda(A(u)))-(n-1)^kC_n^ke^{2ku},\end{align}
and
\begin{align}
\widetilde{F}(u)=F(u)d^{2k}.
 \end{align}
 By \cite{M.Gursky1}, there exists a unique solution $u\in C^\infty(\Omega)$ of \eqref{main}-\eqref{mainbdry}.
 Then, $F(u)=0$ in $\Omega.$

 Our goal in this section is to  derive boundary expansion for $u$ involving all local terms by the maximum principle in a sufficiently small neighborhood of $\partial \Omega.$

Denote by $(x',d)$ the principle coordinates near boundary and by
$\kappa_1, \cdots, \kappa_{n-1}$ the principal curvatures of $\partial \Omega$, with respect to
the interior unit normal vectors. We set
$$H_{\partial\Omega}=\kappa_1+\cdots+ \kappa_{n-1}.$$

In the following, we calculate the operator $F$ on polynomials of $d$ and always use $C$ for a positive constant depending only on $\Omega,$ $n$ and $k.$ We set
\begin{align}\label{formal-solution}
    v=-\log d+c_0+c_1d+c_2d^2+\cdot\cdot\cdot+c_{n,1}d^n\log d,
\end{align}
where $c_i$ $i=1,\cdot\cdot\cdot,n-1,$ and $c_{n,1}$ are functions of $x'$ to be determined.

\begin{lemma}
Let $\Omega$ be a bounded smooth domain in $\mathbb{R}^{n},$ for $n\ge 3.$ Then, there exist functions
$c_1$, $\cdots$, $c_{n-1},$ and $c_{n,1}$ defined on $\partial\Omega$
such that, for $v$ defined in \eqref{formal-solution},
\begin{align}\label{dn}
\widetilde{F}(v)=O(d^{n+\tau}),
\end{align}
where $\tau$ is an arbitrarily given positive constant in $(0,1).$
\end{lemma}

\begin{proof} For $v$ as in \eqref{formal-solution}, set
\begin{align*}
\widetilde{A}_{ij}=\widetilde{A}_{ij}(v)=d^2A_{ij}(v).
\end{align*}
A straightforward calculation yields
\begin{align}\label{Atil}
\begin{split}
    \widetilde{A}_{aa}&=(n-1)+(n-2)\kappa_ad+H_{\partial\Omega}d-2(n-2)c_1d\\
    &\qquad+\text{polynomial of $d$ with power higher than } 2,\\
    \widetilde{A}_{an}&=\widetilde{A}_{ab}=\text{polynomial of $d$ with power higher than } 2,\\
    \widetilde{A}_{nn}&=(n-1)+H_{\partial\Omega}d+\text{polynomial of $d$ with power higher than }2.
   \end{split}
\end{align}
We now substitute $v$ in
$$\widetilde{F}(v)=\sigma_k(\lambda(\widetilde{A}(v)))-d^{2k}(n-1)^kC_n^ke^{2kv},$$
and arrange $\widetilde{F}(v)$ in an ascending order of $d.$ By requiring the constant term and the
coefficient of $d$ to be zero in $\widetilde{F}(v)$, we have
$$c_0=0,$$
and
\begin{align}\label{c1}
c_1=\frac{1}{2(n-1)}H_{\partial\Omega}.
\end{align}
For $i=2,\cdot\cdot\cdot,n-1,$ by requiring the coefficient of $d^{i},$ $i=2,\cdot\cdot\cdot,n-1,$ to be zero in $\widetilde{F}(v)$ successively, we have
\begin{align}\label{ci}
    c_i=\frac{1}{(i-n)(i+1)}G_i(c_1,c_2,\cdot\cdot\cdot,c_{i-1}),
\end{align}
where $G_i$ is a smooth function in $c_1,\cdot\cdot\cdot,c_{i-1}$ and their derivatives.
We point out that the coefficient of $d^n\log d$ in $\widetilde{F}(v)$ equals $0.$
By requiring the coefficient of $d^{n}$ to be zero in $\widetilde{F}(v),$ we have
\begin{align}\label{cn1}
    c_{n,1}=\frac{1}{(n-1)(n+1)}G_n(c_1,c_2,\cdot\cdot\cdot,c_{n-1}),
\end{align}
where $G_i$ is a smooth function in $c_1,\cdot\cdot\cdot,c_{n-1}$ and their derivatives.
Then we obtain the desired result.
\end{proof}

The functions  $c_1,\cdots,c_{n-1}$ and $c_{n,1}$ defined in \eqref{c1}, \eqref{ci} and \eqref{cn1} are functions
on $\partial\Omega$. They are  the coefficients of the so-called local terms, since
they can be expressed explicitly in terms of principal curvatures and their derivatives.
For a demonstration, we calculate $c_2$.

\begin{prop}\label{2parts-coefficient}
The function $c_2$ in \eqref{ci} has the following expression:
\begin{align}\label{c2}\begin{split}
    c_2&=\frac{n}{6(n-2)}\left\{\big(\frac{-3n+2}{4n(n-1)^2}-\frac{n^3-3n-n^2+4}{2n(n-1)^4}\big)
    H_{\partial\Omega}^2+\big(\frac{2}{n}+\frac{(n-2)^2}{2n(n-1)^3})\big|\Pi|^2\right\}
    \\&\qquad+k\big(-\frac{(n-2)}{12(n-1)^3}\big)\big|\overset{\circ}{\Pi}\big|^2,\end{split}
\end{align}
where $\Pi$ is the second fundamental form and
$\overset{\circ}{\Pi}$ is the trace-free second fundamental form, i.e.,
\begin{align*}
    \overset{\circ}{\Pi}=\Pi-\frac{1}{n}Hg.
\end{align*}
\end{prop}

We note that $c_2$ can be expressed as the sum of two parts,
the first part independent of $k$ and the second part a conformal invariant multiplied by $k$.

\begin{proof}
By \eqref{ci} and \eqref{c1}, we have
$$\frac{6}{n}(n-2)c_2=I+\widehat{I},$$ where
\begin{align*}
    I&=
    -\frac{1}{(n-1)^3n}\bigg\{\sum_{a<b}(\frac{1}{n-1}H_{\partial\Omega}+(n-2)\kappa_a)(\frac{1}{n-1}H_{\partial\Omega}\\
    &\qquad\qquad\quad+(n-2)\kappa_b)+(n-1)H_{\partial\Omega}^2\bigg\}
+\frac{-3n+2}{4n(n-1)^2}H_{\partial\Omega}^2+\frac{2}{n}|\Pi|^2,\\
\widehat{I}=&\frac{k}{(n-1)^3n}\bigg\{\sum_{a<b}(\frac{1}{n-1}H_{\partial\Omega}+(n-2)\kappa_a)(\frac{1}{n-1}H_{\partial\Omega}\\
&\qquad\qquad\quad+(n-2)\kappa_b)+\big((n-1)-\frac{(n-1)n}{2}\big)H_{\partial\Omega}^2\bigg\}.
\end{align*}
Set
$$D=\sum_{a<b}(\frac{1}{n-1}H_{\partial\Omega}+(n-2)\kappa_a)(\frac{1}{n-1}H_{\partial\Omega}+(n-2)\kappa_b)+(n-1)H_{\partial\Omega}^2.$$
Then,
\begin{align*}D&=\frac{n^3-3n-n^2+4}{2(n-1)}H_{\partial\Omega}^2-\frac{(n-2)^2}{2}|\Pi|^2,\\
\big|\overset{\circ}{\Pi}\big|^2&=\sum_a(\kappa_a-\frac{1}{n-1}H_{\partial\Omega})^2=|\Pi|^2-\frac{1}{n-1}H_{\partial\Omega}^2,\end{align*}
and
$$D-\frac{(n-1)n}{2}H_{\partial\Omega}^2=-\frac{(n-2)^2}{2}(|\Pi|^2-\frac{1}{n-1}H_{\partial\Omega}^2)=-\frac{(n-2)^2}{2}\big|\overset{\circ}{\Pi}\big|^2.$$
Hence, we have \eqref{c2}.
\end{proof}

Before deriving boundary expansion for $u,$ we show a version of the maximum principle,
which will be of use to us.

\begin{theorem}\label{max-prin}
Let $\Omega$ be a bounded smooth domain in $\mathbb{R}^{n},$ for $n\ge 3.$
     Suppose $u$ and $v$ are smooth sub- and super-solutions, respectively, to \eqref{main}-\eqref{mainbdry} in $\Omega$ and $A(v)\in \Gamma_{k}^{+}.$ If $\lim_{x\rightarrow \partial \Omega}(u-v)\leq 0,$
      then $u\leq v$ in $\Omega.$
\end{theorem}

\begin{proof}
Suppose that $u>v$ somewhere in $\Omega$. Let $C$ be the maximum of $u-v$, which is attained at some point $x_0\in\Omega.$ Then, $w=u-C$ is a strict sub-solution to (\ref{main1}). Hence at the point $x_0,$ we have
$w(x_0)=v(x_0)$ and
$F(w)(x_0)>F(v)(x_0).$
Then,
\begin{align*}
\sigma_k(\lambda(A(w)))(x_0)>\sigma_k(\lambda(A(v)))(x_0).
\end{align*}
However, $v\geq w$ near $x_0.$ Therefore, we have
$dw(x_0)=dv(x_0)$ and $(v-w)_{ij}(x_0)\geq 0,$
and hence
  $A(w)(x_0)\leq A(v)(x_0).$
We use Lemma 3.1 in \cite{M.Gursky2} and then obtain
\begin{align*}
  \sigma_k(\lambda(A(w)))(x_0)\leq\sigma_k(\lambda(A(v)))(x_0).
\end{align*}
This leads to a contradiction.
\end{proof}

According to Theorem 1.4 in \cite{M.Gursky1}, the solution $u$ to \eqref{main}-\eqref{mainbdry} has the decay estimate \eqref{G-expansion}.
Now, we prove that the decay rate is actually $O(d)$.

\begin{lemma}\label{e-d1} Assume that $\Omega$ is a bounded smooth domain in $\mathbb{R}^{n},$ for $n\ge 3$
     and that $u$ is the solution of  \eqref{main}-\eqref{mainbdry}. Then,
 \begin{align}\label{const-term}
   |u+\log d|\leq C d    \quad\text{in }D_{\delta_1},
 \end{align}
 where $C$ and $\delta_1$ are positive constants depending only on $\Omega,$ $n$ and $k.$
 \end{lemma}

\begin{proof}
By (\ref{G-expansion}), we can take a small positive constant $\varepsilon$ to be determined and then a small enough positive constant $\delta_0$ depending on $\varepsilon$ such that
    \begin{align}\label{1}
   |u+\log d|\leq \varepsilon\quad\text{in }D_{\delta_0}.
 \end{align}
 Set
 $$\phi=-\log d+Cd.$$
    Take a small positive constant $\delta_1<\delta_0$ to be determined and set
    \begin{align}\label{Cd1}
        C= \frac{\varepsilon}{\delta_1}.
    \end{align}
    Then by (\ref{1}), (\ref{Cd1}) and (\ref{G-expansion}), we have
    \begin{align*}
        u\leq -\log d+Cd \quad\text{on } \partial D_{\delta_1},
    \end{align*}
    and
    \begin{align}\label{epsilon}
        Cd\leq\varepsilon \quad\text{in } D_{\delta_1}.
    \end{align}
   We use principle coordinates in $ D_{\delta_1}.$
    By (\ref{Atil}) and (\ref{epsilon}), we have, when $\varepsilon\ll1$ and $\delta_1$ are small,
  \begin{align*}
    F(\phi)&=\frac{1}{d^{2k}}\sigma_k(\lambda(d^{2}A(\phi)))-d^{2k}(n-1)^kC_n^ke^{2k\phi})\\
    &=\frac{1}{d^{2k}}\{(n-1)^{k-1}C_{n-1}^{k-1}[(2(n-1))H_{\partial\Omega}d-2(n-2)(n-1)Cd]\\
    &\qquad-(n-1)^{k}C_{n}^{k}2kCd+O(\varepsilon C d)\}.
  \end{align*}
  Hence, when $\delta_{1}$ and $\varepsilon$ are small enough and thus $C$ big enough, we have $F(\phi)<0$ in $D_{\delta_1}.$ Here, by the definition of $C$ in \eqref{Cd1}, we know that the choices of $\delta_1$ and $\varepsilon$ are independent.
  On the other hand,
  by (\ref{Atil}) and (\ref{epsilon}), when $\delta_1$ and $\varepsilon$ are small enough, we have
  \begin{align*}
        \sigma_k(\lambda(A(\phi)))&=\frac{1}{d^{2k}}(\sigma_k(\lambda(d^{2}A(\phi)))>0.
    \end{align*}
Obviously, $A(\phi)\in \Gamma_k^+.$ Therefore, by using Theorem \ref{max-prin}, we have $u\leq\phi=-\log d+Cd$ in $D_{\delta_1}.$ Similarly, we can prove $u\geq-\log d-Cd$ in $D_{\delta_1}.$
\end{proof}

Now, we can derive the boundary expansion for $u$ involving all local terms.

\begin{proof}[Proof of Theorem \ref{expansion}]
Take $\delta_2$ small to be determined such that $\delta_2\leq\delta_1$, where $\delta_1$ is as in Lemma \ref{e-d1}. Consider in $D_{\delta_2}.$ For any fixed $\varepsilon\in(0,1)$, set
\begin{align}\label{defA}
\begin{split}
     A&=2C\delta_2^{1-n},
     \\q&=n+\varepsilon,
\end{split}
\end{align}
where $C$ is a large enough constant depending on the constant in (\ref{const-term}) and $\partial\Omega,n,k.$ By the definition of $A,$ when $\delta_1$ is small,
\begin{align*}
   Ad^n-Ad^q \geq\frac{A}{2}d^n\geq Cd \quad \text{on }\partial D_{\delta_2},
\end{align*}
and
\begin{align}\label{mu}
    Ad^n\leq 2C\delta_2 \quad \text{in } D_{\delta_2}.
\end{align}
Hence, for a positive constant $\mu\ll1$ to be determined, we can choose $\delta_2$ small such that
\begin{align}\label{Adn}
    Ad^n\leq \mu.
\end{align}
Next, set $$\varphi=Ad^n-Ad^q ,$$ and
  \begin{align*}
  \overline{v}&=v+\varphi,\\
  \underline{v}&=v-\varphi.
  \end{align*}
where $c_i,$ $\cdots$, $c_{n-1},$ and $c_{n,1}$ are the functions on $\partial \Omega$
and $v$ is defined in \eqref{formal-solution}.
Then by \eqref{dn} and \eqref{Adn}, a straightforward calculation yields, in $ D_{\delta_{2}},$
\begin{align}
\widetilde{F}(\overline{v})=
-C_{n-1}^{k-1}(n-1)^k(2\varepsilon)(n+1+\varepsilon)Ad^{n+\varepsilon}+O(\mu Ad^{n+\varepsilon}).
\end{align}
Choose $\delta_2$ small enough and thus $\mu$ small by \eqref{mu} and $A$ large by (\ref{defA}). Then,
$\widetilde{F}(\overline{v})<0$ and therefore $F(\overline{v})=\frac{1}{d^{2k}}\widetilde{F}(\overline{v})<0$ in $D_{\delta_{2}}.$
Next by (\ref{Adn}),
we have
   $ A(\overline{v})\in \Gamma_k^+$, if $\delta_{2}$ is small.
By the maximum principle Theorem \ref{max-prin},  $u\leq \overline{v}$ in $D_{\delta_{2}}.$
Similarly, we have $u\geq \underline{v}$ in $D_{\delta_{2}}.$
Hence, we have the desired result.
\end{proof}

Next, we prove the conformal invariance of the difference of the coefficients in expansions
as described in Proposition \ref{conformal-diff}.

\begin{proof}[Proof of Proposition \ref{conformal-diff}]
For $g=e^{2\rho}g_0,$ as discussed in Section 1,
 \begin{align}\label{M.G1}
    \sigma_k[-g^{-1}e^{-2\widetilde{w}_k}Ric(e^{2\widetilde{w}_k}g)]=\beta_{k,n}
\end{align} is equivalent to \begin{align}
    \sigma_k[-g_0^{-1}e^{-2(\rho+\widetilde{w}_k)}Ric(e^{2(\rho+\widetilde{w}_k)}g_0)]=\beta_{k,n}.
\end{align}
Hence, $w_k=\rho+\widetilde{w}_k.$ Assume that $\rho$ has the expansion
$$\rho=\rho_0+\rho_1d+\cdot\cdot\cdot+\rho_{n-1}d^{n-1}+O(d^n),$$
where $\rho_1, \rho_2, \cdots$ are functions on $\partial\Omega$.
By the expansion in Theorem \ref{expansion}, we have,
for $j=1,\cdot\cdot\cdot,n-1,$
\begin{align*}
 \widetilde{c}_j^k-\widetilde{c}_j^1=(c_j^k-\rho_j)- (c_j^1-\rho_j)= c_j^k-c_j^1.
\end{align*}
This is the desired result.
\end{proof}

\section{The $C^1$-Estimates}

In this section, we prove $C^1$ estimate for $u+\log d$ in a sufficiently small neighborhood of $\partial \Omega$ where $u$ is the solution to \eqref{main}-\eqref{mainbdry}.

\begin{lemma}\label{nabla-const}
Assume that $\Omega$ is a bounded smooth domain in $\mathbb{R}^{n},$ for $n\ge 3$,
     and that $u$ is the solution of  \eqref{main}-\eqref{mainbdry}. Then
\begin{align*}
    |\nabla(u+\log d -c_1d)|\leq C\quad\text{in }D_{\delta_3},
\end{align*}
 where $c_1$ is the function in \eqref{c1}, and $C$ and $\delta_3$ are positive constants depending only on $\Omega,$ $n$ and $k$.
\end{lemma}

\begin{proof}
Take $\delta_2$ as the constant in Theorem \ref{expansion} and $\widetilde{c_1},\,\psi\in C^\infty( \Omega)$ satisfying
\begin{align*} \widetilde{c_1}=c_1,\,\psi= d\quad\text{in}\quad D_{\frac{1}{2}\delta_2},\end{align*}
and
\begin{align*}
  \psi\geq \frac{1}{2}\delta_2 ,\quad\text{in } \Omega\setminus D_{\frac{1}{2}\delta_2},
\end{align*}
where $c_1$ is the function as given in \eqref{c1}.
Set
$$w=u+\log \psi -\widetilde{c_1}\psi.$$
We will prove for some $C_0>1,$
\begin{align}\label{w/d2}
 |\frac{w}{\psi^2}|\leq C_0\quad\text{in } \Omega.
\end{align}
First, by Theorem \ref{expansion}, we know (\ref{w/d2}) holds in $D_{\delta_2/2}.$
Next, take
$$j_1=-\log(\frac{1}{2}\delta_2)+C\delta_2,\quad j_2=-\log(\frac{1}{2}\delta_2)-C\delta_2.$$
By Remark 4.10 in \cite{M.Gursky1}, for $i=1,2$, respectively, we can solve
\begin{align}\label{sub-sup}
F(u_{j_i})&=0\quad\text{in}\quad\Omega\setminus D_{\frac{1}{2}\delta_2},\\
u_{j_i}&=j_i\quad\text{on}\quad\partial( \Omega\setminus D_{\frac{1}{2}\delta_2}).
\end{align}
By maximum principle and Lemma \ref{e-d1}, we obtain
 $u_{j_2}\le u\le u_{j_1}$  in $\Omega\setminus D_{\delta_2}/2$.
 Hence, (\ref{w/d2}) holds in $\Omega\setminus D_{\frac{1}{2}\delta_2}.$

We rewrite the equation  \eqref{main} as
\begin{align}\label{new-main}
\sigma_k(\psi^2(\overline{A}(w-\log \psi +\widetilde{c_1}\psi)))&=e^{2k\widetilde{c_1}\psi}(\frac{n-1}{n-2})^kC_n^ke^{2kw}\doteq e^{2k\widetilde{c_1}\psi}\beta_{n,k}e^{2kw}\quad \text{in } \Omega,
\end{align}
where
 \begin{align}\label{Aijbar}
 (\overline{A}(u))_{ij}=\partial_{ij}u+\frac{1}{n-2}\Delta u\delta_{ij}+|\nabla u|^2\delta_{ij}-\partial_{i}u\partial_{j}u.
\end{align}
We denote the $(k-1)$-Newton transformation associated with $\psi^2\overline{A}(w-\log \psi +\widetilde{c_1}\psi)$ as $T_{k-1}\doteq T$, which is positive since $\psi^2\overline{A}\in \Gamma_k^+.$ In particular, if $A^{i}_{j}$
are the components of a symmetric matrix $A$, then the $q$th Newton transformation
associated with $A$ is
$$T_q(A)^{i}_{j}=\frac{1}{q!}\delta^{i_1i_2...i_qi}_{j_1j_2...j_qj}A^{j_1}_{i_1}\cdots A^{j_q}_{i_q}.$$
Here $\delta^{i_1i_2...i_qi}_{j_1j_2...j_qj}$ is the generalized Kronecker delta symbol. We frequently use the following properties of $T_{k-1}(A)$:
 \begin{align}\label{property}\begin{split}
    T_{k-1}(A)_{ij}A_{ij}&=k\sigma_k(A);\\
    tr T_{k-1}(A)&=(n-k+1)\sigma_{k-1}(A);\\
    \partial_m(\sigma_k(A))&=T_{k-1}(A)_{ij} \partial_m(A_{ij}).\end{split}
\end{align}
Set $$Q_{ij}=T_{ij}+\frac{1}{n-2}T_{ll}\delta_{ij}.$$
There is a summation in $l.$ Then, $Q_{ij}$ is positive definite.
For the definition and properties of Newton transformation, we can refer to \cite{M.Gursky2}.
Set
$$\phi(s)=\frac{1}{p^2(3C_0)^p}(2C_0+ s)^p,$$
 for some $p$ large to be determined and $C_0$ as in (\ref{w/d2}). Then,
 \begin{align*}
   \frac{1}{p^2} \geq \phi(s)>0\quad\text{for any } s \in[-C_0,C_0].
\end{align*}
Set
\begin{align*}
    h=(1+\frac{|\nabla w|^2}{2})e^{\phi( \frac{w}{\psi^2})}\doteq ve^{\phi( \frac{w}{\psi^2})}.
\end{align*}
We will prove,  for some constant $C,$
$$|h|_{L^{\infty}(\Omega)}\leq C.$$ This implies the desired result.

 First, for any point $x_0\in\partial \Omega,$ take the principal coordinates $(x',d)$ at $x_0$ with the unit inner normal vector $\nu$ in the $x_n$-direction. By Theorem \ref{expansion}, we know $w\equiv 0$ on $\partial \Omega$ and $w\le Cd^2$ in $D_{\delta_2}.$ Hence, $\nabla_{x'}w\equiv0$ on $\partial\Omega$ and
\begin{align*}
   | \frac{\partial w}{\partial\nu}(x_0)|=|\lim _{d\rightarrow 0}\frac{w(x_0',d)-0}{d-0}|=0.
\end{align*} Hence, $\nabla w(x_0)=0$, implying $|h(x_0)|\leq C.$

Thus, without loss of generality, we can assume that the maximum of $h$ attains at a point $x_0\in\Omega.$ The proof is inspired by \cite{M.Gursky2}. Assume $|\nabla w(x_0)|$ is sufficiently large. Otherwise the conclusion is immediate. All the calculation below is at the point $x_0.$ For brevity, we write
$$s=\frac{w}{\psi^2}.$$
Differentiate $h$ twice. Since $Q_{ij}$ is positive definite, we have
$$h_i=0,\quad
Q_{ij}h_{ij}\frac{\psi^4}{ve^\phi}\le 0.$$ Hence,
\begin{align}\label{wliwl0}
w_{li}w_l=-v\phi'(s)(\frac{w}{\psi^2})_i,
\end{align} and
\begin{align}\label{hijqij}\begin{split}
 \frac{\psi^4}{v} Q_{ij} w_{lij}w_l+(\phi''( s)-\big(\phi' (s)\big)^2) Q_{ij}(\frac{w}{\psi^2})_i(\frac{w}{\psi^2})_j\psi^4 +\phi' (s) Q_{ij}(\frac{w}{\psi^2})_{ij}\psi^4\le 0.\end{split}
\end{align}
 By (\ref{w/d2}), we have
\begin{align*}
  \ \psi^4\partial_i( \frac{w}{\psi^2}) \partial_j (\frac{w}{\psi^2})&=
  w_iw_j+O(|\nabla w|\psi)+\frac{4w^2\psi_i\psi_j}{\psi^2},\\ \psi^4\partial_{ij}( \frac{w}{\psi^2})&=w_{ij}\psi^2+O(|\nabla w|\psi+\psi^2).
\end{align*}
We will prove later $\phi''(s)-(\phi'(s))^2>0.$
Then, (\ref{hijqij}) reduces to
\begin{align}\label{new-main-1o}\begin{split}
    0&\geq\frac{1}{v} Q_{ij} w_{lij}w_l\psi^4+(\phi''( s)-(\phi' (s))^2) Q_{ij}(w_iw_j+O(|\nabla w|\psi+1))\\
    &\qquad+\phi' (s) Q_{ij}(w_{ij}\psi^2+O(|\nabla w|\psi+1)).\end{split}
\end{align}
By the properties in (\ref{property}), we have
\begin{align}\label{qijwij}\begin{split}
    Q_{ij}(w_{ij}\psi^2)&=T_{ij}(\psi^2\overline{A}(w-\log \psi +\widetilde{c_1}\psi)_{ij}
    -\psi^2|\nabla w|^2\delta_{ij}\\&\qquad+\psi^2\partial_{i}w\partial_{j}w+O(|\nabla w|\psi+1))\\
   & =k\beta_{n,k}e^{2k\widetilde{c_1}\psi}e^{2kw}\\
   &\qquad+T_{ij}(-\psi^2|\nabla w|^2\delta_{ij}+\psi^2\partial_{i}w\partial_{j}w+O(|\nabla w|\psi+1)).
\end{split}\end{align}
Next, by applying $\partial_m$ to (\ref{new-main}), we obtain
\begin{align}\label{diff1}\begin{split}
&T_{ij}\big(2\psi \partial_{m}\psi(\partial_{ij}w+\frac{1}{n-2}\Delta w\delta_{ij}+|\nabla w|^2\delta_{ij}-\partial_{i}w\partial_{j}w)\\
&\quad+\psi^2(\partial_{ijm}w+\frac{1}{n-2}\Delta w_m\delta_{ij}-2v\phi'(s)\partial_m(\frac{w}{\psi^2})\delta_{ij}-\partial_{i}w\partial_{jm}w
-\partial_{im}w\partial_{j}w)
\\&\quad -2\partial_{lm}w(\partial_{l}\psi\psi-\partial_l(\widetilde{c_1}\psi)\psi^2)\delta_{ij}+2\partial_{jm}w(\partial_{i}\psi\psi-\partial_i(\widetilde{c_1}\psi)\psi^2)+O(1+|\nabla w|)\big)\\
&=2k\beta_{n,k}e^{2k\widetilde{c_1}\psi}e^{2kw}\partial_{m}w+2k\beta_{n,k}e^{2k\widetilde{c_1}\psi}e^{2kw}\partial_{m}(\widetilde{c_1}\psi).
\end{split}\end{align}
We multiply \eqref{diff1} by $\frac{1}{v}\psi^2\partial_m w$ and sum over $m.$ Then by (\ref{property}) and (\ref{wliwl0}), we get
\begin{align}\label{wijmwm}\begin{split}
  \frac{1}{v} Q_{ij} w_{lij}w_l\psi^4&=\frac{2}{v}k\beta_{n,k}e^{2k\widetilde{c_1}\psi}e^{2kw}|\nabla w|^2\psi^2+O(1)\\&\quad+T_{ij}(2\phi'(s)\psi^2|\nabla w|^2\delta_{ij}-2\phi'(s)\psi^2 w_iw_j+O(1+|\nabla w|\phi'(s))).
\end{split}\end{align}
Note $0<\phi'(\frac{w}{\psi^2}),\phi''(\frac{w}{\psi^2})<1$ and substitute (\ref{qijwij}), (\ref{wijmwm}) and
$$Q_{ij}w_{i}w_{j}=T_{ij}w_{i}w_{j}+\frac{1}{n-2}T_{ll}|\nabla w|^2$$
into (\ref{new-main-1o}). Then, we have
\begin{align*}
   0&\geq O(1)+T_{ij}\big((\phi''( s)-(\phi' (s))^2)w_iw_j+(\phi''( s)-(\phi' (s))^2)\frac{1}{n-2}|\nabla w|^2\delta_{ij}\\
   &\qquad+2\phi'(s)\psi^2|\nabla w|^2\delta_{ij}-2\phi'(s)\psi^2 w_iw_j-\phi'(s)\psi^2|\nabla w|^2\delta_{ij}+\phi'(s)\psi^2w_{i}w_{j}\\&\qquad+O(1+|\nabla w|)\big)
   \\&=O(1)+T_{ij}\big((\phi''( s)-(\phi' (s))^2-\phi'(s)\psi^2)w_iw_j\\
   &\qquad+(\frac{1}{n-2}\phi''( s)-\frac{1}{n-2}(\phi' (s))^2+\phi'(s)\psi^2)|\nabla w|^2\delta_{ij}+O(1+|\nabla w|)\big)
\end{align*}
By the expression of $\phi,$  we have, for a large constant $C,$
\begin{align*}
    \phi'(\frac{w}{\psi^2})>\frac{1}{p3^pC_0},\end{align*}
and
\begin{align*}
   \phi''(\frac{w}{\psi^2})-(\phi')^2(\frac{w}{\psi^2})-C\phi'(\frac{w}{\psi^2}) >\frac{1}{p3^pC_0^2}(p-1-\frac{1}{p}-3CC_0).
\end{align*}
Fix $p$ large enough. Then, we have, for some positive $\epsilon,$
\begin{align*}
   C&\geq\epsilon T_{ij} w_iw_j+T_{ij}\big(2\epsilon|\nabla w|^2\delta_{ij}+O(1+|\nabla w|)\big)\\
   &\geq \epsilon T_{ij} w_iw_j+T_{ij}\big(\epsilon|\nabla w|^2\delta_{ij}+O(1)\delta_{ij}\big),
\end{align*}
where we used the fact $|T_{ij}|^2\leq T_{ii}T_{jj}.$
 Take $B$ large to be determined.

{\it Case 1. The matrix $\epsilon|\nabla w|^2\delta_{ij}+O(1)\delta_{ij}$ has an eigenvalue less than $B$.} In this case, the gradient estimate is immediate.

{\it Case 2. The matrix $\epsilon|\nabla w|^2\delta_{ij}+O(1)\delta_{ij}$ has all eigenvalues bigger than $B$.}
By absorbing lower order terms, we have  \begin{align*}
   C\geq\epsilon T_{ij} w_iw_j+T_{ll}B.
\end{align*} By (\ref{property}), we have $\sigma_{k-1}\leq C,$ independent of $B.$ Then by Proposition 4.2 in \cite{M.Gursky2}, \eqref{property} and the positive lower bound for $\sigma_{k}$, we can fix $B$ large enough
to get a contradiction.

Then, we have $|\nabla w|^2(x_0)\leq C.$ This finishes the proof.
\end{proof}

We now improve Lemma \ref{nabla-const} under the same assumption.

\begin{proof}[Proof of Theorem \ref{nabla}]
Take $\alpha$ as in Theorem \ref{nabla}, $\psi $ as in the proof of Lemma \ref{nabla-const} and $\widetilde{c_1},\cdot\cdot\cdot,\widetilde{c}_{n,1}\in C^\infty( \Omega)$ satisfying
\begin{align*}
  \widetilde{c_1}=c_1,\cdot\cdot\cdot,\widetilde{c}_{n,1}=c_{n,1}\quad\text{in}\quad D_{{\delta_2}/{2}},
\end{align*}
where $c_i,$ $i=1,\cdot\cdot\cdot,n-1,$ and $c_{n,1}$ are functions as in (\ref{c1}), (\ref{ci}) and (\ref{cn1}) and we rewrite the constant $\delta_3$ in Lemma \ref{nabla-const} as $\delta_2$.

Set
\begin{align}\label{f}f=\widetilde{c_1}\psi+\cdot\cdot\cdot+\widetilde{c}_{n,1}\psi^n\log\psi,\end{align} and
$$w=u+\log \psi -f.$$
First, we will prove, for some $C_0>1,$
\begin{align}\label{w/dn}
|\frac{w}{\psi^{2+2\alpha}}|\leq C_0\quad\text{in } \Omega.
\end{align}
By Theorem \ref{expansion}, (\ref{w/dn}) holds in $D_{\delta_2/2}.$ We point out that,
in order to apply Theorem \ref{expansion}, we require $2+2\alpha\le n$, which results in the
choice of $\alpha$ in the statement of Theorem \ref{nabla}.
Next, using $u_{j_1}$ and $u_{j_2},$
 obtained in the proof of Lemma \ref{nabla-const}, we know that \eqref{w/dn} holds in $\Omega\setminus D_{\delta_2/2}.$

We rewrite the equation  \eqref{main} as
\begin{align}\label{new-main-1}
\sigma_k(\psi^2(\overline{A}(w-\log \psi +f)))&=e^{2kf}(\frac{n-1}{n-2})^kC_n^ke^{2kw}\doteq e^{2kf}\beta_{n,k}e^{2kw}\quad \text{in } \Omega,
\end{align}
where $(\overline{A}(u))_{ij}$ is as in \eqref{Aijbar} and $f$ is as in \eqref{f}.
We use $T_{k-1}\doteq T$ for $(k-1)$-Newton transformation associated with $\psi^2\overline{A}(w-\log \psi +f)$, which is positive since $\psi^2\overline{A}\in \Gamma_k^+.$
Set $$Q_{ij}=T_{ij}+\frac{1}{n-2}T_{ll}\delta_{ij}.$$ Then, $Q_{ij}$ is positive definite by \cite{M.Gursky2}. By the properties in (\ref{property}), we have
\begin{align*}
    Q_{ij}(w_{ij}\psi^2)&=T_{ij}(\psi^2\overline{A}(w-\log \psi +f)_{ij}-\psi^2|\nabla w|^2\delta_{ij}\\
    &\qquad+\psi^2\partial_{i}w\partial_{j}w+O(|\nabla w|\psi+1))\\
    &=k\beta_{n,k}e^{2kf}e^{2kw}+T_{ij}(-\psi^2|\nabla w|^2\delta_{ij}+\psi^2\partial_{i}w\partial_{j}w+O(|\nabla w|\psi+1)),\end{align*}
and hence
\begin{align}\label{qijwij-1} Q_{ij}w_{ij} >T_{ij}(-|\nabla w|^2\delta_{ij}+\partial_{i}w\partial_{j}w+O(\frac{|\nabla w|}{\psi}+\frac{1}{\psi^2})).
\end{align}
Set
$$\phi(s)=\frac{1}{p^2(3C_0)^p}(2C_0+ s)^p,$$
for some $p$ large to be determined and $C_0$ in (\ref{w/dn}). Then,
 \begin{align*}
   \frac{1}{p^2} \geq \phi(s)>0\quad\text{for any }s \in[-C_0,C_0].
\end{align*}
Set
\begin{align*}
    h=(1+\frac{1}{2}|\nabla(\frac{ w}{\psi^{\alpha}})|^2)e^{\phi( \frac{w}{\psi^{2+2\alpha}})}\doteq ve^{\phi( \frac{w}{\psi^{2+2\alpha}})}.
\end{align*}
We will prove, for some constant $C,$
$$|h|_{L^{\infty}(\Omega)}\leq C.$$
This would imply the desired conclusion.

First, for an arbitrary point $x_0\in\partial \Omega,$ we can argue similarly as in the proof of Lemma \ref{nabla-const}. Note that $w$ defined in this proof satisfies $\frac{w}{\psi^{\alpha}}\equiv 0$ on $\partial \Omega$ and $\frac{w}{\psi^{\alpha}}\le Cd^{2+\alpha}$ in $D_{\delta_2}.$ Then, $\nabla_{x'}(\frac{w}{\psi^{\alpha}})\equiv0$ on $\partial\Omega$ and
\begin{align*}
   | \frac{\partial }{\partial\nu}(\frac{w}{\psi^{\alpha}})(x_0)|=|\lim _{d\rightarrow 0}\frac{(\frac{w}{\psi^{\alpha}})(x_0',d)-0}{d-0}|=0.
\end{align*} Hence, $\nabla (\frac{w}{\psi^{\alpha}})(x_0)=0$, implying $|h(x_0)|\leq C.$

Thus, without loss of generality, we can assume that the maximum of $h$ attains at a point $x_0\in\Omega.$ The proof is inspired by \cite{M.Gursky2}. Take $A$ large to be determined. Without of generality, we assume $|\nabla (\frac{w}{\psi^{\alpha}})(x_0)|\geq A$ is sufficiently large. Otherwise the conclusion is obvious.
All calculation below is at $x_0.$ For brevity, we write
$$s=\frac{w}{\psi^{2+2\alpha}}.$$
By differentiating $h$ once, we have $h_i=0$ and hence
\begin{align}\label{wliwl}
(\frac{w}{\psi^{\alpha}})_{li}(\frac{w}{\psi^{\alpha}})_l=-v\phi'(s)(\frac{w}{\psi^{2+2\alpha}})_i.
\end{align}
Using (\ref{w/dn}), we have
\begin{align*}
  (\frac{w}{\psi^{\alpha}})_i&=\frac{w_i}{\psi^{\alpha}}+O(\psi^{\alpha+1}),\\
   \partial_{ij}( \frac{w}{\psi^{\alpha}})&=\frac{w_{ij}}{\psi^{\alpha}}+O(\frac{|\nabla w|}{\psi^{\alpha+1}}),
  \\ \partial_{ij}( \frac{w}{\psi^{2\alpha+2}})&=\frac{w_{ij}}{\psi^{2\alpha+2}}+O(\frac{|\nabla w|}{\psi^{2\alpha+3}}+\frac{1}{\psi^2}).
\end{align*}
Apply $\partial_m$ to (\ref{new-main-1}) and then by Lemma \ref{nabla-const}, we have
\begin{align}\label{diff1-1}\begin{split}
    Q_{ij}(\frac{w}{\psi^{\alpha}})_{ijm}&=T_{ij} \{  -2 (\frac{w}{\psi^{\alpha}})_{lm}(\frac{w}{\psi^{\alpha}})_{l}\psi^{\alpha}\delta_{ij} +(\frac{w}{\psi^{\alpha}})_{im}(\frac{w}{\psi^{\alpha}})_{j}\psi^{\alpha}+(\frac{w}{\psi^{\alpha}})_{jm}(\frac{w}{\psi^{\alpha}})_{i} \psi^{\alpha} \\&+(\frac{w}{\psi^{\alpha}})_{im}O(\frac{1}{\psi})+O(|\nabla(\frac{w}{\psi^{\alpha}})|\frac{1}{\psi^{2}}+\frac{1}{\psi^{3+\alpha}}
    +|\nabla(\frac{w}{\psi^{\alpha}})|^2\frac{1}{\psi^{1-\alpha}})\} \\&+Q_{ij}w_{ij}O(\frac{1}{\psi^{1+\alpha}})+O(\frac{1}{\psi^{3+\alpha}}+|\nabla(\frac{w}{\psi^{\alpha}})|\frac{1}{\psi^{2}}).
\end{split}\end{align}
Next, differentiate $h$ one more time. Since $Q_{ij}$ is positive definite,
we have $0\geq Q_{ij}h_{ij}\frac{1}{ve^\phi}$ and hence
\begin{align}\label{hijqij-1}\begin{split}
   0&\geq\frac{1}{v} Q_{ij} (\frac{w}{\psi^{\alpha}})_{lij}(\frac{w}{\psi^{\alpha}})_l+(\phi''( s)-\big(\phi' (s)\big)^2) Q_{ij}(\frac{w}{\psi^{2+2\alpha}})_i(\frac{w}{\psi^{2+2\alpha}})_j\\
   &\qquad+\phi' (s) Q_{ij}(\frac{w}{\psi^{2+2\alpha}})_{ij}.\end{split}
\end{align}
We sum (\ref{diff1-1}) with $\frac{1}{v}(\frac{w}{\psi^{\alpha}})_m.$
Note $0<\phi'(\frac{w}{\psi^2}),\phi''(\frac{w}{\psi^2})<1.$  Then,
\begin{align}\label{wijmwm-1}\begin{split}
    &\frac{1}{v}Q_{ij}(\frac{w}{\psi^{\alpha}})_{ijm}(\frac{w}{\psi^{\alpha}})_m\\
    &\quad=T_{ij} \{2 \phi'(s)\psi^{2+2\alpha} ( |\nabla(\frac{w}{\psi^{2+2\alpha}})|^2\delta_{ij} -(\frac{w}{\psi^{2+2\alpha}})_i(\frac{w}{\psi^{2+2\alpha}})_{j}) \\
    &\qquad+\phi'(s)|\nabla(\frac{w}{\psi^{2+2\alpha}})|O(\frac{1}{\psi})+O(\frac{1}{\psi^{2}}+\frac{1}{\psi^{3+\alpha}}\frac{1}{A}
    +|\nabla(\frac{w}{\psi^{\alpha}})|\frac{1}{\psi^{1-\alpha}})\} \\
    &\qquad+Q_{ij}w_{ij}O(\frac{1}{\psi^{1+\alpha}}\frac{1}{A})+O(\frac{1}{\psi^{3+\alpha}}\frac{1}{A}+\frac{1}{\psi^{2}}).
\end{split}\end{align}
  Note $\phi'(\frac{w}{\psi^2})>\frac{1}{p3^pC_0}$ and we will prove later $\phi''(s)-(\phi'(s))^2>0.$
Then by (\ref{qijwij-1}), (\ref{wijmwm-1}) and Lemma \ref{nabla-const}, (\ref{hijqij-1}) reduces
\begin{align}\label{new-main-2}\begin{split}
   0&\geq O(\frac{1}{\psi^{3+\alpha}}\frac{1}{A}+\frac{1}{\psi^{2}})\\
   &\qquad+T_{ij} \bigg\{((\phi''( s)-\big(\phi' (s)\big)^2)\frac{1}{n-2}  +2 \phi'(s)\psi^{2+2\alpha} )|\nabla(\frac{w}{\psi^{2+2\alpha}})|^2\delta_{ij} \\
   &\qquad+((\phi''( s)-\big(\phi' (s)\big)^2) -2 \phi'(s)\psi^{2+2\alpha})(\frac{w}{\psi^{2+2\alpha}})_i(\frac{w}{\psi^{2+2\alpha}})_{j} \\
   &\qquad+\phi'(s)|\nabla(\frac{w}{\psi^{2+2\alpha}})|O(\frac{1}{\psi})+O(\frac{1}{\psi^{4+2\alpha}}
    )\bigg\}.
   \end{split}
\end{align}
Multiply \eqref{new-main-2} by $\psi^{4+2\alpha}.$
By
$$(\frac{w}{\psi^{2+2\alpha}})_i=(\frac{w}{\psi^{\alpha}})_i(\frac{1}{\psi^{2+\alpha}})+O(\frac{1}{\psi}),$$
we have
\begin{align}\label{new-main-11}
\begin{split}
   0&\geq O(1)+T_{ij} \bigg\{((\phi''( s)-\big(\phi' (s)\big)^2)\frac{1}{n-2}  +2 \phi'(s)\psi^{\alpha} )|\nabla(\frac{w}{\psi^{\alpha}})|^2\delta_{ij} \\&\qquad+((\phi''( s)-\big(\phi' (s)\big)^2) -2 \phi'(s)\psi^{\alpha})(\frac{w}{\psi^{\alpha}})_i(\frac{w}{\psi^{\alpha}})_{j}\\&\qquad+|\nabla(\frac{w}{\psi^{\alpha}})|O(1)+O(1
    )\bigg\}.
   \end{split}
\end{align}
By the expression of $\phi,$ for a large constant $C,$ we have
\begin{align*}
    \phi'(\frac{w}{\psi^2})&>\frac{1}{p3^pC_0},\\
   \phi''(\frac{w}{\psi^2})-(\phi')^2(\frac{w}{\psi^2})-C\phi'(\frac{w}{\psi^2}) &>\frac{1}{p3^pC_0^2}(p-1-\frac{1}{p}-3CC_0).
\end{align*}
Fix $p$ large enough. Then, we have, for some positive $\epsilon,$
\begin{align*}
   C\geq\epsilon T_{ij} (\frac{w}{\psi^{\alpha}})_i(\frac{w}{\psi^{\alpha}})_j+T_{ij}\big(2\epsilon|\nabla(\frac{w}{\psi^{\alpha}})|^2+O(1)\big)\delta_{ij},
\end{align*}
where we used the fact $|T_{ij}|^2\leq T_{ii}T_{jj}.$
 Take $B$ large to be determined and we consider two cases.

{\it Case 1.} If the matrix $$2\epsilon|\nabla(\frac{w}{\psi^{\alpha}})|^2\delta_{ij}+O(1)\delta_{ij}$$ has an eigenvalue less than $B$, then the gradient estimate is immediate.

{\it Case 2.} Otherwise, absorbing lower order terms, we have
\begin{align*}
   C\geq\epsilon T_{ij} (\frac{w}{\psi^\alpha})_i(\frac{w}{\psi^\alpha})_j+BT_{ll}.
\end{align*}
We argue similarly as in the proof of Lemma \ref{nabla-const}.
Then we have $|\nabla (\frac{w}{\psi^\alpha})|(x_0)\leq C.$
\end{proof}

\begin{remark}\label{alpha-n}
We emphasize again that the validity of (\ref{w/dn}) requires a  relation of $\alpha$ and $n.$ In fact, for a general $\alpha\ge \frac{1}{2}$ and $w$ defined above, when $n\ge 2+2\alpha,$ we have
\begin{align*}
|\frac{w}{\psi^{2+2\alpha}}|\le C,\quad\quad
|\frac{\nabla w}{\psi^{\alpha}}|\le C.
\end{align*}
\end{remark}

\section{The $C^2$-Estimates}

In this section, we derive estimates of second derivatives.

\begin{proof}[Proof of Theorem \ref{C2}]
Take $w$, $\psi$ and $f$ as defined in the proof of Theorem \ref{nabla}. This proof is divided into two steps.

\emph{Step 1.} We will prove that there exists a constant $C$, depending only on $\partial \Omega,$ $n$ and $k$, such that
$$\Delta w\ge -C\quad\text{in }\Omega.$$
We proceed to prove this in $D_{\delta_3/2}$, where $\delta_3$ is the constant in Theorem \ref{nabla}.
The proof in $\Omega\setminus D_{\delta_3/2}$ is similar but easier.

By \eqref{main}-\eqref{mainbdry}, Theorem \ref{expansion} and noting $\psi=d$ in $D_{\frac{1}{2}\delta_3},$ we have, in $D_{\frac{1}{2}\delta_3},$
\begin{align}\label{sk}
\begin{split}
  \frac{n}{(C_n^k)^{1/k}} \bigg(\sigma_k\big(\lambda(\frac{1}{n-2}A(u))\big)\bigg)^{1/k}&=\frac{1}{d^2}(\frac{n(n-1)}{n-2})e^{2(w+f)}
  \\&=\frac{1}{d^2}(\frac{n(n-1)}{n-2})(1+2c_1d+O(d^2)).
  \end{split}
   \end{align}
By the expression of $A(u)$ in \eqref{Aij}, a straightforward calculation yields, in $D_{\frac{1}{2}\delta_3},$
   \begin{align}\label{s1}
   \begin{split}
  \sigma_1\bigg(\lambda\big(\frac{1}{n-2}A(u)\big)\bigg)&=(1+\frac{n}{n-2})\Delta u+(n-1)|\nabla u|^2\\
  &=(1+\frac{n}{n-2})\Delta w+\frac{1}{d^2}(\frac{n(n-1)}{n-2})+\frac{c_1}{d}\frac{2n(n-1)}{n-2}+O(1),
  \end{split}
\end{align}
where we used the fact that $-\Delta d=H_{\partial\Omega}+O(d)$ and the definition of $c_1$ in \eqref{c1}. By Maclaurin's inequality,
we have
\begin{align*}
 \sigma_1\bigg(\lambda\big(\frac{1}{n-2}A(u)\big)\bigg)\ge  \frac{n}{(C_n^k)^{1/k}} \sigma_k^{1/k}\bigg(\lambda\big(\frac{1}{n-2}A(u)\big)\bigg).
\end{align*}
By combining with \eqref{sk} and \eqref{s1} and by a straightforward calculation, we have $\Delta w\ge -C.$

\emph{Step 2.} Next, we will prove
$$\max_{\gamma\in \mathbb S^{n-1},p\in \overline{\Omega}}\partial_\gamma\partial_\gamma w\leq C,$$
where $C$ is a positive constant depending only on $\Omega,$ $n$ and $k.$
The proof of this step is inspired by \cite{M.Gursky2}.

First, assume $n\ge 2+2\alpha,$ for some $\alpha\ge 1.$
Later on, we will take $\alpha=3$ but we write it in the present form to demonstrate why we choose $\alpha=3.$ By Remark \ref{alpha-n}, we have
\begin{align}\label{8}
\begin{split}
|\frac{w}{\psi^{2+2\alpha}}|\le C,\quad\quad
|\frac{\nabla w}{\psi^{\alpha}}|\le C.
\end{split}
\end{align}
Hence, $\nabla w\equiv 0$ on $\partial \Omega.$ Moreover, in principal coordinates at any
boundary point $x_0$ with $e_n$ as the unit inner normal vector to $\partial \Omega$ at $x_0,$ we have
\begin{align*}
    \nabla_{x'} \nabla w(x_0)= 0,\end{align*}
and
\begin{align*}
    \\ |\nabla_n\nabla w(x_0)|\le \lim_{d\rightarrow0}|\frac{Cd-0}{d}|=C.
\end{align*}
Therefore, we obtain
\begin{align}\label{C2bdry}
  |\nabla^2 w|_{L^\infty(\partial \Omega)}\le C.
\end{align}
Next, set
\begin{align}\label{h}
    h(p,\gamma)=\partial_\gamma\partial_\gamma w(p)+\Lambda\frac{|\nabla w|^2}{\psi^{2\alpha}}(p)\quad
    \text{for } (p,\gamma)\in \overline{\Omega}\times \mathbb{S}^{n-1},
\end{align}
where $\Lambda$ is a constant to be determined.
We will prove
$$|h|_{L^\infty( \overline{\Omega}\times \mathbb{S}^{n-1})}\le C,$$
which implies the conclusion in Step 2 by \eqref{8}.

Without loss of generality, we  assume that the maximum of $h$ attains at $(p, \gamma)\in\Omega\times \mathbb{S}^{n-1}.$ Otherwise, by \eqref{C2bdry}, the conclusion is immediate. Then, by rotating coordinates at $p$, we may assume $\frac{\partial}{\partial x_1}=\gamma$. Set
\begin{align*}
\widetilde{h}(x)=h(x,\frac{\partial}{\partial x_1})=w_{11}+\Lambda\frac{|\nabla w|^2}{\psi^{2\alpha}}.
\end{align*}
Without loss of generality, we can assume $w_{11}(p)\ge 1.$ Otherwise, the desired result is immediate.
Since $p$ is the maximum point of $\widetilde{h},$
we have, at $p,$
\begin{align}\label{tiltahi}
0=\partial_i\widetilde{h}=w_{11i}+\Lambda\frac{2w_kw_{ki}}{\psi^{2\alpha}}-2\alpha\Lambda\frac{\psi_i|\nabla w|^2}{\psi^{2\alpha+1}},
\end{align}
and
\begin{align}\label{tiltahij}
\begin{split}
0\ge \partial_{ij}\widetilde{h}&=w_{11ij}+\Lambda\frac{2w_{kj}w_{ki}}{\psi^{2\alpha}}
 -4\alpha\Lambda\frac{w_{kj}w_{k}\psi_i}{\psi^{2\alpha+1}}
  -4\alpha\Lambda\frac{w_{ki}w_{k}\psi_j}{\psi^{2\alpha+1}}\\
  &\qquad-2\alpha\Lambda\frac{\psi_{ij}|\nabla w|^2}{\psi^{2\alpha+1}}
  +2\alpha(2\alpha+1)\Lambda\frac{\psi_{i}\psi_{j}|\nabla w|^2}{\psi^{2\alpha+2}}
  +\Lambda\frac{2w_{k}w_{kij}}{\psi^{2\alpha}}.
  \end{split}
\end{align}
All the calculation below is at $p.$
Recall from Section 3 that
$$Q_{ij}=T_{ij}+\frac{1}{n-2}T_{ll}\delta_{ij}.$$
Since $Q_{ij}$ is positive definite, using \eqref{8} and \eqref{tiltahij}, we have, at $p,$
\begin{align}\label{tiltahijqij}
0\ge Q_{ij}w_{11ij}+\Lambda\frac{2}{\psi^{2\alpha}}Q_{ij}w_{kj}w_{ki}
+\Lambda O(\frac{1}{\psi^{\alpha}})Q_{ij}w_{kij}+\Lambda T_{ij}\delta_{ij}O(\frac{1}{\psi^2}
+\frac{|\nabla^2 w|}{\psi^{\alpha+1}}).
\end{align}
We write  $\overline{A}_{ij}= \overline{A}_{ij}(u)$ for convenience.

First, we consider the term $Q_{ij}w_{kij}.$
By \eqref{property} and \eqref{8}, differentiating \eqref{new-main-1} with respect to $x_m$, we have  \begin{align*}
 \partial_m (\text{L.H.S.})&= \partial_m \sigma_k(\psi^2\overline{A}_{ij}) =T_{ij}(2\psi\psi_m\overline{A}_{ij}+\psi^2(\overline{A}_{ij})_m)=\psi^2 T_{ij}(\overline{A}_{ij})_m+k\sigma_k\frac{2\psi_m}{\psi},\\
  \partial_m (\text{R.H.S.})&=\partial_m (e^{2k(f+w)}\beta_{n,k})=O(1).
\end{align*}
Hence,
\begin{align}\label{TijAijm}
T_{ij}(\overline{A}_{ij})_m=O(\frac{1}{\psi^3}).
\end{align}
On the other hand, substituting $u=-\log \psi+w+f$ in  $\overline{A}_{ij}(u)$ and then by \eqref{8}, we have
\begin{align}\label{qijwijm}
Q_{ij}w_{ijm}=O(\frac{1}{\psi^3})(1+T_{ij}\delta_{ij})+O(\frac{|\nabla^2 w|}{\psi})T_{ij}\delta_{ij}.
\end{align}

Next, we consider the term $Q_{ij}w_{11ij}.$
Set $\sigma=(\sigma_k)^{1/k}.$ Then
\begin{align}\label{sigma}
\sigma(\lambda(\psi^2\overline{A}_{ij}))=e^{2(f+w)}\beta_{n,k}^{1/k}.
\end{align}
Differentiate \eqref{sigma} twice with respect to $x_1$ and compare the R.H.S. with the L.H.S. By \eqref{8} and the concavity of $\sigma,$ we have
\begin{align}\label{TijAij11}
T_{ij}(\psi^2\overline{A}_{ij})_{11}\ge -C-Cw_{11}.
\end{align}
Substituting $u=-\log \psi+w+f$ in $\overline{A}_{ij}(u)$ and then by \eqref{property},
\eqref{8}, \eqref{tiltahi} and \eqref{TijAijm}, we have
\begin{align}\label{qijw11ij}
Q_{ij}w_{11ij}\ge -\frac{C}{\psi^2}-\frac{C}{\psi^2}w_{11}+
T_{ij}\delta_{ij}O(\frac{1}{\psi^4}+|\nabla ^2 w|^2+\frac{|\nabla ^2 w|}{\psi^{\alpha+1}}).
\end{align}
Substitute \eqref{qijwijm} and \eqref{qijw11ij} into \eqref{tiltahijqij}. Then, multiply
\eqref{tiltahijqij} by $\psi^{3+\alpha}.$ By \eqref{8}, we have
\begin{align*}
0 &\ge -C-C\psi^{\alpha+1}w_{11}+
T_{ij}\delta_{ij}(\frac{1}{\psi^{\alpha-3}}\frac{2 \Lambda}{n-2}\sum_{i,k}|w_{ki}|^2+O((1+\Lambda)+\psi^{\alpha+3}|\nabla ^2 w|^2\\&\qquad+(1+\Lambda)|\nabla ^2 w|\psi^2)+\frac{2 \Lambda}{\psi^{\alpha-3}}T_{ij}w_{kj}w_{ki}.
\end{align*}
Choose $\Lambda$ large enough and, without loss of generality, we may assume $\sum_{i,k}|w_{ki}|^2$ is large and much larger than $\Lambda$. Then we have, for a positive constant $c,$
\begin{align}\label{bd-for-sigmak-1}
C+C\psi^{\alpha+1}w_{11}\ge \sigma_{k-1}(\psi^2\overline{A}_{ij})c\sum_{i,k}|w_{ki}|^2\ge \sigma_{k-1}(\psi^2\overline{A}_{ij})cw_{11}^2.
\end{align}
On the other hand, by Maclaurin's inequality, we have
\begin{align*}
    \sigma_{k-1}\ge  ( \frac{\sigma_{k} }{\binom{n}{k}})^{\frac{k-1}{k}}\binom{n}{k-1}.
\end{align*}
Note that, for some positive $c_0,$ $$\sigma_k(\psi^2\overline{A}_{ij})=e^{2k(f+w)}\beta_{n,k}>c_0,$$
where we used \eqref{8} and the definition of $f.$ Then, we have ,for some positive $c_1,$
$$\sigma_{k-1}(\psi^2\overline{A}_{ij})>c_1.$$
Hence, \eqref{bd-for-sigmak-1} implies, for some positive constant $\epsilon _0,$
\begin{align*}
C+Cw_{11}\ge \epsilon _0w_{11}^2.
\end{align*}
Then, we draw the conclusion $w_{11}\le C$ and finish the proof in Step 2.

Combining the two steps, we have the desired conclusion.
\end{proof}

\end{document}